\documentclass[final]{IEEEtran}

\usepackage{amsthm}
\usepackage{braket,amsfonts,amsopn}  
\usepackage{dsfont}  
\usepackage{xcolor}
\usepackage{amsmath,amssymb}
\usepackage{hyperref}
\usepackage{circuitikz}
\usepackage{graphicx}

\usepackage[acronym]{glossaries}

\usepackage{comment}
\usepackage{verbatim}

\newcommand{\be}{\begin{equation}}
\newcommand{\ee}{\end{equation}}

\newcommand {\R}{\mathbb R}

\newcommand {\C}{\mathbb C}

\renewcommand{\Re}{\operatorname{Re}}
\renewcommand{\Im}{\operatorname{Im}}
\newcommand{\trace}{\operatorname{tr}}
\newcommand{\diag}{\operatorname{diag}}

\newacronym{rfm}{RFM}{Ribosome Flow Model}
\newacronym{pmp}{PMP}{Pontryagin Maximum Principle}
\newacronym{tas}{TASEP}{Totally Asymmetric Simple Exclusion Process}

\title{On the gain of entrainment in stable  linear control systems with a nonlinear output\thanks{RM and MM are with the School of ECE, Tel Aviv University, Tel Aviv 69978, Israel.   
This research   is partially supported by
the Israel Science Foundation (grant number 221/24).  
Correspondence: michaelm@tauex.tau.ac.il 
}}

 \author{Ram Massas  and 
 Michael~Margaliot}

\newtheorem{Theorem}{Theorem}
\newtheorem{lemma}[Theorem]{Lemma}
\newtheorem{proposition}[Theorem]{Proposition}
\newtheorem{corollary}[Theorem]{Corollary}

\newtheorem{Remark}[Theorem]{Remark}

\newtheorem{example}[Theorem]{Example}
\AtBeginEnvironment{example}{%
  \pushQED{\qed}%
}
\AtEndEnvironment{example}{\popQED\endexample}

\begin{document}

	\maketitle

\begin{abstract}
   A control system admits a positive gain of entrainment~(GOE) if entrainment to a periodic input yields a larger output, on average, than the output generated by the corresponding constant input with the same mean value. We analyze~GOE in continuous-time stable linear control systems with a static  nonlinear output map. Although linear systems with linear outputs   have zero~GOE, we show that a nonlinear output may generate a nontrivial GOE through the mismatch between the average output along the entrained periodic orbit  and the output evaluated at the corresponding averaged equilibrium.

We derive  a second-order characterization of GOE for smooth output maps 
revealing  that the leading-order contribution is determined by the curvature of the output map. We then show that if the output is convex (concave) on the controllable subspace, then GOE is nonnegative (nonpositive) for every periodic input. Furthermore, GOE admits a natural geometric interpretation as the average Bregman divergence between the entrained periodic orbit and the equilibrium associated with the averaged input.

For the special case of quadratic output functions, we derive explicit  frequency-domain formulas for~GOE. These yield necessary and sufficient conditions guaranteeing the sign of GOE, characterize the contribution of individual input harmonics, and lead to an optimal periodic excitation that maximizes~GOE under an energy constraint. The theoretical results are illustrated using an electrical RLC circuit and a compartmental pharmacodynamic model with a nonlinear drug-effect map. 

\end{abstract}



\section{Introduction}

Consider the   nonlinear control system
\begin{align}\label{eq:nonlinsys}
    \dot{x}(t) &= f(x(t),u(t)),  \\
    y(t) &= h(x(t),u(t)), \nonumber
\end{align}
where \(x(\cdot)\in\mathbb{R}^n\) is the state,
\(u(\cdot)\in\mathbb{R}^m\) is the input, and~\(y(\cdot)\in\mathbb{R}\) is a scalar output. 
We say that the system \emph{entrains} to periodic excitations if, for every $T$-periodic control input~$u$, the system admits a unique $T$-periodic solution~$\gamma^u$ that is globally attracting. In other words, $x(t,u,x_0)$ converges to~$\gamma^u$ for any initial condition~$x_0$.

Entrainment is an important property  in numerous engineering and biological systems. In electrical engineering, synchronous generators must entrain to the frequency of the power grid. In neuroscience, populations of neurons   synchronize with periodic external stimuli. In biology,   cellular and physiological processes entrain to periodic environmental or internal signals like  the 24-hour solar day or the  periodic cell-division cycle.   

Entrainment is not a generic property of  nonlinear control systems, as periodic excitation may lead to complex dynamical behaviors including multi-stability, subharmonic oscillations, and  chaos~\cite{Takac1992_subharmonic,subharmincs_sontag_2018}. One important class of nonlinear systems for which entrainment can be established rigorously is the class of contractive systems~\cite{sontag_cotraction_tutorial,strom1975logarithmic}, 
 and this has
 found many applications~\cite{entrain_trans,entrainME,RFM_entrain}.

Assume from here on that our goal is to maximize (in some appropriate sense) the scalar output~$y$. 
Let \(v:[0,\infty ) \to\mathbb{R}^m\) be a nonconstant \(T\)-periodic input, and denote
its average value by
\begin{equation}
    \bar v := \frac{1}{T}\int_0^T v(t)\,dt .
\end{equation}
If the system~\eqref{eq:nonlinsys}
entrains, then in response to the control~$u(t)=v(t)$
every solution will converge
to a unique~$T$-periodic orbit~$\gamma^v$, whereas for the 
constant control~$u(t)= \bar v$ every solution will converge to a unique  equilibrium~$e ^{\bar v} $. In the first case, the output will converge to the $T$-periodic  pattern~$ h(\gamma^v(t),v(t))$, and in the second to the constant value~$h(e^{\bar v},\bar v)$. 
The \emph{gain of entrainment}~(GOE) for the control~$v$ is:
\[
\text{GOE}(v): = 
\frac1{T}\int_0^T h(\gamma^v(t),v(t) ) \, dt -h(e^{\bar v},\bar v) 
\]
(see Figure~\ref{fig:goe_schematic}).
If this quantity is positive then entrainment does not only yield  synchronization to the periodic excitation~$v$, but also a larger output, on average.

Note that in the definition of GOE  we make a ``fair comparison'' by considering the effect of  two controls~$u(t)=v(t)$  and~$u(t)=
\bar v$, which share the same time average.

GOE may be important in numerous applications including
periodic fishery, periodic production of proteins, 
and periodic medication protocols. 
To illustrate the idea, consider traffic flow regulated by a sequence of $T$-periodic traffic lights. Assume that the vehicle dynamics converge to a $T$-periodic regime. It is then natural to ask whether, by appropriately coordinating the periodic switching patterns of the lights, one can achieve a higher average flow than that obtained under constant (time-invariant) signal settings with the same average green-time allocation.

\begin{figure}[t]
    \centering
    \includegraphics[width=0.95\linewidth]{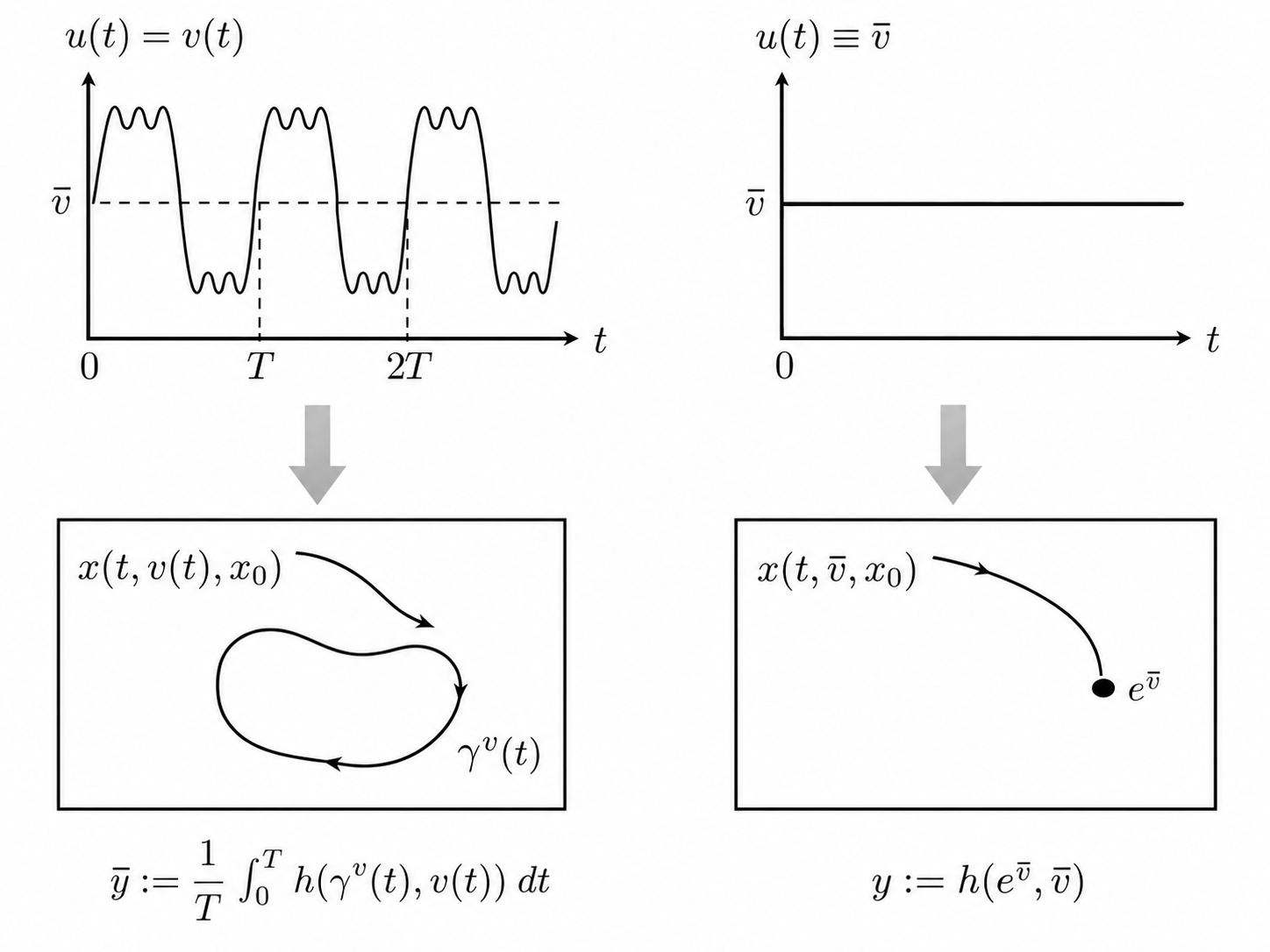}
    \caption{Gain of entrainment~(GOE) in a system with a scalar  output~\(y=h(x,u)\). We compare the effect of two controls with the same average
    value. Left: \(u(t)=v(t)\) is a \(T\)-periodic control,  \(\gamma^v(t)\) 
  is the corresponding globally attracting periodic
    trajectory, 
    and~\(\bar y\) is the average value of the output 
    along~\(\gamma^v\). Right: \(u(t)\equiv \bar v\), where~\(\bar v:=\frac{1}{T}\int_0^T v(t)\,dt\), is the corresponding constant
    control and~\(e^{\bar v}\) is the corresponding globally attracting equilibrium.
    The system admits a positive~GOE for the control~\(v\) if~\(\bar y>h(e^{\bar v},\bar v)\).}
    \label{fig:goe_schematic}
\end{figure}

 Another important  application is photosynthesis, i.e. the process of converting light radiation to  chemical energy. Natural light radiation is   periodic  with a period of 24h, and in photobioreactors light can be artificially modulated in a time-periodic manner. 
As noted in~\cite{Nedbal2026_photo_syn}: ``Photosynthetic yield, productivity, and stress resilience depend not only on the mean intensity of
photosynthetically active radiation, but also on its temporal variability.''
Thus, a natural question is whether    a periodic  light radiation pattern 
increase the average productivity  in photosynthesis.

Since constant controls are included as a special case of periodic controls,
one might expect that periodic operation  always   improves performance. However, this is not necessarily true. 
For linear control systems with a linear output  the 
average output generated by a periodic input is equal to the output generated
by the corresponding constant averaged input. 
Therefore, such systems have zero~GOE.  For an important nonlinear model from systems biology, known as the ribosome flow model~\cite{margaliot2012stability,Raveh2016}, GOE has been analyzed in several special cases, and evidence suggests that it is always negative~\cite{GOE_RFM_2023}.
 The recent paper~\cite{katz2023gainentrainmentclassweakly} showed that, for a broad class of nonlinear systems, GOE is inherently a high-order effect with respect to the control amplitude. More specifically, when the amplitude of the periodic excitation is small, the resulting change in the average output appears only in the second- or higher-order terms of the amplitude expansion, whereas the  linear term vanishes. These results suggest that~GOE may be more subtle and elusive than initially expected.

In this note, we study GOE in a system
with a linear 
state dynamics, but
nonlinear   output. More precisely, we consider the   multiple-input
single-output~(MISO) linear control system:
\begin{align}\label{eq:intro_miso_system}
    \dot{x}(t) &= Ax(t)+Bu(t), \\
    y(t) &= h(x(t)), \nonumber
\end{align}
where \(x(\cdot)\in\mathbb{R}^n\), \(u(\cdot)\in\mathbb{R}^m\),
\(y(\cdot)\in\mathbb{R}\),   and
\(h:\mathbb{R}^n\to\mathbb{R}\) is a nonlinear output map.  The  set  of admissible controls~$\mathbb U$ is the set of   measurable and essentially bounded functions~$u:[0,\infty)\to\R^m$. 
We assume throughout that~$A$ is Hurwitz, so the system entrains.

This class of systems 
is sometimes called   Wiener systems, i.e., a linear dynamical  system followed by a static nonlinear output map (see e.g. \cite{wiener_systems_book}).
They 
arise  naturally in many applications. For example, in
linear electrical circuits, the state may consist of capacitor voltages  and
inductor currents, while the output may represent  the
stored 
energy, which is a quadratic function of the state. Similarly, in mechanical
systems, displacement and velocity may obey linear mass-spring-damper
dynamics, whereas the total mechanical energy is the sum of quadratic kinetic
and elastic potential energies. Other examples include 
linear systems with nonlinear performance criteria, and nonlinear systems with a
linearized   dynamics  in the vicinity of an operating point.

 This note analyzes GOE in stable Wiener systems. We show that, although linear systems with linear outputs   exhibit zero~GOE, a nonlinear output can generate a nontrivial GOE through the mismatch between the output evaluated along the entrained periodic trajectory and the output evaluated at the corresponding averaged equilibrium.  We   derive a second-order characterization of GOE for smooth output functions, showing that the first-order contribution vanishes and that the leading-order behavior is determined by the curvature of the output map. For convex (concave) outputs, we prove that GOE is always nonnegative (nonpositive), and we further show that GOE admits a natural interpretation as an average Bregman divergence. We then focus on the case of  quadratic outputs, and  obtain explicit state-space and frequency-domain formulas that reveal how individual harmonics contribute to GOE and provide conditions guaranteeing its sign. 
 The theoretical results are illustrated  using two applications. The first is a series RLC circuit driven
by a periodic voltage source. The second is a linear compartmental pharmacodynamic model with a nonlinear dose-response map.    

We use standard notation. Vectors [matrices] are denoted by small [capital] letters. The transpose of a matrix~$A$ is~$A^\top$. 
The trace  of a square matrix~$A$ is~$\trace(A)$. We will use the cyclic property of the trace operator:
\[
\trace(AB)=\trace(BA) 
\]
for any~$A\in\C^{n\times m} $ and~$B\in\C^{m\times n}$.
The $n\times n$ identity matrix is~$I_n$. 
We use~$\mathrm{i}$ to denote the imaginary unit.
For a complex number~$z\in \C$,  $\Re(z)$ [$\Im(z)$] is the real  
[imaginary] part of~$z$.
The polar representation of the  complex number~$a+\mathrm{i}b$ is~$r\exp(\mathrm{i}\theta)$ with~$r= (a^2+b^2)^{1/2}$ and~$\theta=\arctan(b/a)$.
For a complex matrix~$Z\in\C^{n\times m} $, we use~$Z^*$ to denote the conjugate transpose 
of~$Z$. For a $T$-periodic function~$z:[0,\infty)\to\R^\ell$, we use~$\bar z$ to denote the time-average of~$z$, that is,~$\bar z:=\frac1{T}\int_0^T z(t)\; dt$. 

 \section{Main Results}

Fix a nonconstant $T$-periodic control~$v\in\mathbb U$, and denote its  time average by~$\bar v$. Then for the control~$u(t)=v(t)$ any solution of~\eqref{eq:intro_miso_system}
converges to a unique bounded $T$-periodic orbit~$\gamma^v$, whereas
for the control~$u(t)=\bar v$ any solution of~\eqref{eq:intro_miso_system}
converges to an equilibrium~$e^{\bar v}$. 
 
 Let
\[
\mathcal C := \operatorname{range}\{B,AB,\dots,A^{n-1}B\}
\]
denote the controllable subspace of~\eqref{eq:intro_miso_system} (see, e.g.,~\cite[Chapter~3]{sontag_book}).
Since 
\[
x(t)=\exp(At)x(0)+\int_0^t \exp(A(t-s))Bu(s) \, ds,
\]
and~$A$ is Hurwitz, we have that
\be\label{eq:inst}
e^{\bar v} \in \mathcal C \text{ and }\gamma^v(t) \in \mathcal C \text{  for all }t\geq 0.
\ee

We begin with a simple but useful result showing
that 
for a  stable linear system, the time average  of the entrained periodic orbit equals the equilibrium corresponding to the averaged input.

\begin{lemma}\label{lemma:simple}
   We have 
\begin{equation}
    \frac{1}{T}\int_0^T \gamma^v(t)\,dt
    =
    e^{\bar v}.
\end{equation}
\end{lemma}

Geometrically, the lemma implies that if one constructs a convex ``ball'' whose boundary contains the   trajectory~$\gamma^v(t)$, $t\in[0,T)$,  then~$e^{\bar v}$ necessarily lies in   this ball.

\begin{IEEEproof}
Along the periodic solution, we have 
\begin{equation}
    \dot{\gamma}^v(t)=A\gamma^v(t)+Bv(t),
\end{equation}
and averaging  over one period gives
\[
    \frac1{T} \int_0^T \dot{\gamma}^v(t)\,dt
    =\frac1{T} 
    A\int_0^T \gamma^v(t)\,dt
    +\frac1{T} 
    B\int_0^T v(t)\,dt.
\]
Since \(\gamma^v\) is \(T\)-periodic,
$
    \int_0^T \dot{\gamma}^v(t)\,dt
    =
    \gamma^v(T)-\gamma^v(0)
    =
    0$, so 
  $  0
    = 
    A   \overline{\gamma^v} 
    +  
    B  \overline{v}$,
 and thus 
\begin{equation}
    \overline{\gamma^v}
    =
    -A^{-1}B\bar v
    .
\end{equation}
This holds for any~$T$-periodic control and, in particular, for the constant control~$u(t)=\bar v$, so we conclude that
\[
 \overline {    e^{\bar v}  } =    e^{\bar v}\nonumber \\
 = -A^{-1}B\bar v ,
\]
and this completes the proof. 
\end{IEEEproof}

\subsection{Second-order perturbation analysis of   GOE}
It is useful to view the~$T$-periodic control  as a perturbation added to a constant control.
      Consider the control
\begin{equation}
    u_\varepsilon(t) =\bar v+\varepsilon w(t),
    \label{eq:eps_input}
\end{equation}
where~$\varepsilon\in\R\setminus\{0\}$,
$\bar v\in\R^m\setminus\{0\}$, 
and~$ w:[0,\infty)\to\mathbb{R}^m $
is \(T\)-periodic with zero time average, that is,
$
    \frac{1}{T}\int_0^T w(t)\,dt=0.
$   \begin{proposition}\label{prop:perturb}
Assume that the output function~$h$ is~$C^2$.
Let~$\gamma^w$ denote the unique \(T\)-periodic solution of~\eqref{eq:intro_miso_system}
corresponding to the $T$-periodic control~$u=w$. 
  Then
\begin{align}
\text{GOE}(u _ \varepsilon) 
&= \frac{\varepsilon^2}
{2 T}\int_0^T  (\gamma^w(t) )^\top     \nabla^2 h(e^{\bar v}) \gamma^w(t) \, dt  +o(\varepsilon^2)\nonumber \\
&=\frac{\varepsilon^2}{2}\trace((\nabla^2 h(e^{\bar v}) ) \Sigma^{w})+o(\varepsilon^2),\label{eq:prop2_second}
\end{align}
where for any \(T\)-periodic control~\(u\), the matrix
\(\Sigma^{u}\in\R^{n\times n}\) is defined by
\be\label{eq:mat_E}
\Sigma^{u}:= \frac{1}{T}\int_0^T  
    \left(\gamma^{u}(t)-e^{\bar u}\right)
    \left(\gamma^{u}(t)-e^{\bar u}\right)^\top 
    \,dt.
\ee
\end{proposition}

Note that $\Sigma^{u}$ is the ``covariance matrix'' of the entrained periodic orbit~$\gamma^u$
around its mean \(e^{\bar u}\).
       Proposition~\ref{prop:perturb}  shows in particular
       that   the first-order term in~$\varepsilon
       $ in the~GOE   
vanishes, so for
 a small~$\varepsilon$
the sign of  
GOE is determined by the curvature of the output function integrated along the direction~$\gamma^\omega (t)$
which, as shown in the proof below, is the direction~$\gamma^{ u_\varepsilon}(t)-e^{\bar v} $.

\begin{IEEEproof}
Write the periodic solution    as
$
\gamma^{u_\varepsilon}(t)=e^{\bar v}+\varepsilon z(t)  
$,  that is, 
\[
z(t):=\frac1{\varepsilon}(\gamma^{u_\varepsilon}(t)-e^{\bar v}).
\]
Then~$z$ is~$T$-periodic,  and
\begin{align*}
     \varepsilon \dot z &=  A\gamma^{u_\varepsilon} +B u_\varepsilon\\
     &=
     A (\varepsilon z+e^{\bar v}) +B (\bar v+\varepsilon w)\\
     &=\varepsilon A z +\varepsilon B w, 
\end{align*}
so  \(z=\gamma^w\).

A
 Taylor series expansion gives
\begin{align}\label{eq:taylpo}
h(\gamma^{u_\varepsilon}) &= h( e^{\bar v} + \varepsilon z )\nonumber\\
&= h(e^{\bar v}) + 
\varepsilon (\nabla h( e^{\bar v}) )^\top z + \frac{\varepsilon^2}{2} z^\top \nabla^2 h(e^{\bar v}) z+o(\varepsilon^2),
\end{align}
and  
\begin{align*}
\text{GOE} &(u_\varepsilon)=  \frac1{T}\int_0^T h(\gamma^{u_\varepsilon} (t) )\, dt -h(e^{\bar v})\\
&= 
\frac1{T}\int_0^T
\bigl( 
\varepsilon (\nabla h(e^{\bar v}) )^\top z(t) + \frac{\varepsilon^2}{2} z^\top (t)\nabla^2 h(e^{\bar v}) z(t) \bigr) \, dt \\&+o(\varepsilon^2) .
\end{align*}
By the definition of~$z$ and Lemma~\ref{lemma:simple},  $z$ has  time average zero, and using this yields the first equation in~\eqref{eq:prop2_second}. To prove the second equation, note that
\begin{align*}
    \int_0^T z^\top (t)\nabla^2 h(e^{\bar v}) z(t) \, dt & =
\int_0^T    \trace\bigl (z^\top (t)\nabla^2 h(e^{\bar v}) z(t) \bigr)\, dt\\
&= 
\int_0^T    \trace\bigl ( \nabla^2 h(e^{\bar v}) z(t)z^\top (t) \bigr ) \, dt \\
&=
\trace\bigl ( \nabla^2 h(e^{\bar v})
\int_0^T      z(t)z^\top (t)   \, dt
\bigr ),
\end{align*}
where we used the cyclic property of the trace operator. 
\end{IEEEproof}

\subsection{Convexity, Bregman divergence,  and  GOE}
The analysis above   suggests that convexity [concavity] of $h$ is related to the sign of~GOE. The next result formalizes this. 

\begin{proposition}\label{prop:convex_h}
 Assume that the function~$t\to h(\gamma^u(t))$ is integrable on~$[0,T)$. 
If~$h$ is 
  convex [concave]  when restricted to the controllable subspace~$\mathcal C$ then~$\text{GOE}(v)\geq 0$ 
[$\text{GOE}(v)\leq 0$] for any~$T$-periodic control~$v$.
If convexity [concavity] is replaced  by strict convexity [concavity] then the inequalities    become strict for
  any $T$-periodic control~$v$ such that~$\gamma^v $ is not constant.
\end{proposition}

This implies in particular a robustness property of the~GOE, namely,  if~$h$ is a convex [concave] function then  $\text{GOE}(v)\geq 0$
[$\text{GOE}(v)\leq 0$]
regardless of the exact entries in the matrices~$A$ and~$B$ in~\eqref{eq:intro_miso_system} (as long as~$A$ is Hurwitz). 

Note also that a linear output function is both convex and concave so for such a function we conclude that~$\text{GOE}(v)=0$. 
\begin{IEEEproof}
If \(h\) is
  convex   when restricted to~$\mathcal C$  
then using~\eqref{eq:inst} 
and Jensen's inequality yields 
\begin{align}\label{eq:convx}
\bar  y( v)&:=   \frac{1}{T}\int_0^T h(\gamma^v(t))\,dt\nonumber \\
    & \geq 
    h\left(
    \frac{1}{T}\int_0^T \gamma^v(t)\,dt
    \right)\\
    &=h( \overline{\gamma^v}) .\nonumber
\end{align}
Using Lemma~\ref{lemma:simple} gives 
 $   \bar y(v)\geq h(e^{\bar v})$, so 
 $
    \mathrm{GOE}(v) \geq 0.
$
If \(h\) is
strictly convex   when restricted to~$\mathcal C$ and 
\(\gamma^v\) is not constant   
then  the inequality in~\eqref{eq:convx} becomes strict. 
 A   similar argument proves the results in the concave case.   
\end{IEEEproof}

\begin{example}
    Consider the system~\eqref{eq:intro_miso_system} with~$n=2$, $m=1$, $A=\begin{bmatrix}
        -2&1\\1&-2
    \end{bmatrix}$, $B =\begin{bmatrix}
        1\\1
    \end{bmatrix}$, and~$h(x)=x_1 x_2$. 
The controllable subspace is~$\mathcal C=\text{span}(\begin{bmatrix}
    B & AB 
\end{bmatrix})=\text{span}(\begin{bmatrix}
    1\\1
\end{bmatrix})$. The output function~$h$ is 
not convex nor concave on~$\R^2$, but when restricted to~$\mathcal C$  we have~$h(x)=x_1^2$, 
which is convex. Thus, Proposition~\ref{prop:convex_h} implies that~$\text{GOE}(v)\geq 0$. Only the values of $h$ on $\mathcal C$ matter because every periodic orbit and equilibrium lies in~$\mathcal  C$.
\end{example}

\begin{Remark}
Proposition~\ref{prop:convex_h} shows that convexity of~$h$ implies that~$GOE(v)\geq 0$ for all~$v$. This raises the question   whether convexity of~$h$ is also necessary in order for nonnegativity of~GOE for any~$v$. We will see in Example~\ref{eq:exa_posi_non_concvex} below that this is not the case.
\end{Remark}

Recall that if~\(g:\mathbb R^n\to\mathbb R\) is $C^1$  
then the function~$D_g:\R^n\times\R^n\to\R$
defined by
\begin{equation}\label{eq:bregman}
D_g(x,z):=
g(x)-g(z)-(\nabla g(z))^\top (x-z).
\end{equation}
is called the
Bregman divergence associated with~\(g\)~\cite{bregman1967}.
Geometrically, $D_g(x,z)$
  measures how much the graph of $g$ lies above its tangent plane at~$z$. 

\begin{proposition}
Suppose that~$h \in C^1$. Then 
\[
GOE(v)=\frac1T\int_0^T
D_h(\gamma^v(t),e^{\bar v})\,dt.   
\]
\end{proposition}
In other words,
$\text{GOE}(v)$  is   the average
Bregman divergence between the entrained periodic orbit~\(\gamma^v(t)\)
and the equilibrium~\(e^{\bar v}\) generated by the averaged input. Thus,~GOE is not just a performance criterion,
but also a geometric quantity. 
Since~\(D_h(x,z)\geq 0\) for convex~\(h\) (see~\eqref{eq:bregman}), this representation
provides another proof that convexity of~$h$ implies~\(\mathrm{GOE}(v)\geq 0\).

  \begin{IEEEproof}
Applying~\eqref{eq:bregman}
with~$g=h$, 
$
x=\gamma^v(t)$, and~$z=e^{\bar v}$ gives 
\begin{align*}
D_h(\gamma^v(t),e^{\bar v})
&=
h(\gamma^v(t))-h(e^{\bar v})
-(\nabla h(e^{\bar v}))^\top
\left(\gamma^v(t)-e^{\bar v}\right).
\end{align*}
Averaging over one period and using 
Lemma~\ref{lemma:simple}
yields
\begin{align}\label{eq:goe_BREGMAN}
\frac1T\int_0^T
D_h(\gamma^v(t),e^{\bar v})\,dt
&=
\frac1T\int_0^T h(\gamma^v(t))\,dt
-h(e^{\bar v})  ,
\end{align}
and this completes the proof.
\end{IEEEproof}

  \begin{Remark}
  Consider the simplex
\[\mathbb S:=\{x\in\R^n : x_i>0 \text{ for all } i 
\text{ and } \sum_{i=1}^n x_i=1\},
\]
and define~$h:\mathbb S \to\R$  by~$h(x):=\sum_{i=1}^n x_i\log(x_i)$, i.e., the negative
entropy of~$x$. 
  Then~\eqref{eq:bregman} gives
  \begin{align*}
D_h(x,z)
&=\sum_{i=1}^n x_i\log(x_i/z_i)-\sum_{i=1}^n  x_i+\sum_{i=1}^n  z_i\\
&= \sum_{i=1}^n  x_i\log(x_i/z_i),
\end{align*}
i.e. the Kullback-Leibler~(KL) divergence~$D_{KL}(x,z) $
between~$x$ and~$z$.
  
In the master equation, that has found numerous applications in many fields of science (see, e.g. the monograph~\cite{haag2017modelling}), the state variable~$x_i(t)$ represents the probability that the system is in configuration~$i$ at time~$t$, and in particular~$\sum_{i=1}^n x_i(t)=1$ for all~$t\geq 0$. If we define the output map~$h$  as the negative entropy of the state vector 
then~\eqref{eq:goe_BREGMAN} becomes 
\begin{align*}
\text{GOE}(v)= \frac1{T}\int_0^T 
D_{KL}(\gamma^v(t) \| e^{\bar v}) \, dt,  
\end{align*}
so in this case   GOE is the average KL  divergence   between the periodically forced orbit  and the averaged equilibrium. This may suggest  a possible connection between~GOE and  topics like    entropy production in nonequilibrium statistical mechanics (see e.g.~\cite{kolchinsky2026cycleaffinitywindinglocalize} and the references therein).
  \end{Remark}

\subsection{Quadratic output functions}
From here on 
we  consider the    special   case where the output function 
is a  quadratic form, that is,
\be\label{eq:quad}
h(z)=z^\top Q z,\text{ with }  Q\in\R^{n\times n}.
\ee
 We assume throughout without loss of generality that~$Q$ is symmetric. 
 \begin{corollary}\label{prop:goe_var}
     Consider the system~\eqref{eq:intro_miso_system} with the quadratic output mapping~\eqref{eq:quad}. Then 
 \be\label{eq:goe_whight}   
 \mathrm{GOE}(v)
    =
    \frac{1}{T}\int_0^T
    \left(\gamma^v(t)-e^{\bar v}\right)^\top
    Q
    \left(\gamma^v(t)-e^{\bar v}\right)
    \,dt  . 
\ee
Furthermore,\begin{align}\label{eq:withtrace}
 \mathrm{GOE}(v)
    &=\trace( Q \Sigma^v),
\end{align}
with~$\Sigma^v \in\R^{n\times n}$  defined  in~\eqref{eq:mat_E}. 
 \end{corollary}
Note that~\eqref{eq:goe_whight} implies that $\text{GOE}(v)$ is a  ``weighted variance'' of the periodic orbit.

   \begin{IEEEproof}
In this case,~$\nabla^2 h(x)=2Q$ for all~$x$. The analysis in  Proposition~\ref{prop:perturb} becomes exact with the~$o(\varepsilon^2)$ term
equal to zero, and this completes the proof. 
\end{IEEEproof}

Since~$\gamma^v(t),e^{\bar v} \in \mathcal  C$, 
 we have that~$\gamma^v(t)-e^{\bar v}\in \mathcal C$,  and thus  if~$Q
 $ is positive-definite   [negative-definite]   when restricted to~$\mathcal C$ then~\eqref{eq:goe_whight}   
 implies that $\text{GOE}(v)>0$
 [$\text{GOE}(v)<0$] for any nonconstant periodic solution~$\gamma^v $. 
  
 \subsubsection{Frequency-based Approach}

For the case of a quadratic output function, it is also possible to give a frequency-based analysis of~GOE.
Define~$G:\C
\to\C^{n\times m}$  by 
\be\label{eq:def_Gs}
G(s):=  ( s I_n -A)^{-1}B,
\ee
i.e., the transfer function from the input to the state of~\eqref{eq:intro_miso_system}.

Consider the control  
\be\label{eq:multick}
v(t)=\bar v+\sum_{k\in \mathbb K}  c^k \sin (  k \omega  t+\eta_k),
\ee
with~$\omega>0$, 
$\eta_k\in[0,2\pi)$,
$\bar v,c^k\in\R^m\setminus\{0\}$,
and~$\mathbb K$ is a (finite or infinite) set of positive integers such that~$v $ is well-defined. Note that~$v$  is~$T$-periodic with~$T:=2\pi/\omega$, and with  time average~$\bar v$. 
 
\begin{proposition}\label{prop:impulse}
    Consider the system~\eqref{eq:intro_miso_system} with  
     a quadratic output function~\eqref{eq:quad}. For the 
 control~\eqref{eq:multick}, we have 
\be\label{eq:goe_exp}
GOE(v) = \frac{1}{2} 
    \sum_{k\in\mathbb K}  
(c^k)^\top 
\Re\left( G^*(\mathrm{i}k\omega ) Q  
 G(\mathrm{i}k\omega) \right) c^k.
\ee
\end{proposition}

This provides a spectral decomposition of~GOE, and shows that  each harmonic contributes independently to~GOE. Many practical systems are identified experimentally through frequency response functions without identifying the state space representation, and 
Proposition~\ref{prop:impulse} implies GOE can be estimated directly from such frequency-response measurements.

\begin{IEEEproof}
Let~$\phi_k(t):=\exp(\mathrm{i}(k\omega t+\eta_k))$, so that~$\Im(\phi_k(t))=\sin(k\omega t+\eta_k)$.
 The~$T$-periodic  solution 
  corresponding to~$v$ in~\eqref{eq:multick}
  is
\[
\gamma^v(t)=G(0)\bar v + \sum_{k \in \mathbb K}   \Im (\phi_k(t) G(\mathrm{i} k \omega  ) )
       c^k  .
\]
Eq.~\eqref{eq:mat_E} implies that~$\Sigma^v$
is equal to
\[
  \frac{1}{T}\int_0^T  
    \sum_{k \in \mathbb K} \Im (\phi_k(t) G(\mathrm{i} k \omega  ) )
       c^k  (\sum_{\ell \in \mathbb K}  \Im (\phi_\ell(t) G(\mathrm{i} \ell \omega  ) )
       c^\ell )^\top 
    \,dt,
\]
and using orthogonality of the harmonics yields 
\[
\Sigma^v= \frac{1}{T}\int_0^T  
    \sum_{k  \in \mathbb K } \Im (\phi_k(t) G(\mathrm{i} k \omega  ) )
       H_k   (\Im (\phi_k(t) G(\mathrm{i} k \omega  ) ))
        ^\top 
    \,dt  , 
\]
where~$H_k:=c^k(c^k)^\top.$
Simplifying this using the identities~$\Im(z_1 z_2)=
\Re(z_1)\Im(z_2)
+\Im(z_1)\Re(z_2)$, and 
$\frac1{T}\int_0^T \sin^2(k\omega t+\eta_k )\,dt= \frac1{T}\int_0^T \cos^2(k\omega t+\eta_k )\,dt = 1/2 $
gives
\begin{align*}
\Sigma^v  = 
\frac{1}{2} 
    \sum_{k \in \mathbb K } &\left (
 \Im(G(\mathrm{i}k\omega) )H_k 
   \Im(G^\top(\mathrm{i}k\omega))\right .
   \\& +\left . 
   \Re(G(\mathrm{i}k\omega) )H_k 
   \Re(G^\top(\mathrm{i}k\omega)
   ) \right ).
\end{align*}
 Thus, 
\[
\Sigma^v= \frac{1}{2} 
    \sum_{k \in \mathbb K}
\Re\left(  G(\mathrm{i}k\omega )H_k 
 G^*(\mathrm{i}k\omega) \right ) .
\]
Using~\eqref{eq:withtrace}
gives~$\text{GOE}(v)= \trace( Q \Sigma^v)$, and using the cyclic property of the trace operator  completes the proof. 
\end{IEEEproof}

\begin{example}\label{eq:exa_posi_non_concvex}
    Consider the system~\eqref{eq:intro_miso_system} with~$n=2$, $m=1$, $A=\begin{bmatrix} -1&0\\0&-2\end{bmatrix}$, $B=\begin{bmatrix} 1\\1\end{bmatrix}$, and~$h(x)=x^\top Q x$, with~$Q=\diag(1,-1)$. The controllable subspace is
    \begin{align*}
  \mathcal C &=   \text{span}
    (\begin{bmatrix} B &AB\end{bmatrix})=\text{span}(\begin{bmatrix} 1 &-1\\
1&-2\end{bmatrix} )  =\R^2.
    \end{align*}
Thus,~$h$ is not convex nor concave on~$\mathcal C$. 
A calculation gives~$G(s)=(sI_2-A)^{-1}B=\begin{bmatrix}
    \frac1{s+1}\\\frac1{s+2}
\end{bmatrix}$.
Substituting this in~\eqref{eq:goe_exp} gives
\begin{align*}
  GOE(v)  
 &=\frac1{2}
  \sum_{k\in\mathbb K} \frac{3(c^k)^2}{(1+k^2\omega^2)(4+k^2\omega^2)},
\end{align*}
so~$\text{GOE}(v)> 0$ for any~$T$-periodic control~$v$. Thus, strict convexity of~$h$ is sufficient, but not necessary to guarantee  a positive~GOE for all~$v$.
\end{example}

For the case of a quadratic output, 
Proposition~\ref{prop:impulse} implies the following characterization.
\begin{proposition}\label{prop:iff_for_quad}
    Consider the system~\eqref{eq:intro_miso_system} with  
     a quadratic output function~\eqref{eq:quad}. Then the following two conditions are equivalent:   \begin{enumerate}
     \item $\text{GOE}(v)\geq0 $ for any periodic control~$v$; 
     \item  The~$m\times m$ matrix
  $\Re\left (  G^*(\mathrm{i}s)Q G(\mathrm{i}s)\right) $ is nonnegative-definite  for all~$s> 0$.  
\end{enumerate}
\end{proposition}
\begin{IEEEproof}
    If 
    condition~2) holds then~\eqref{eq:goe_exp} implies that condition~1) holds. To prove the converse implication, assume that
    condition~2) does not hold, i.e. there exists a value~$\omega>0 $   and a  vector~$c^1\in\R^m\setminus\{0\}$ such that~$(c^1)
^\top\Re(G^*(\mathrm{i}\omega)QG(\mathrm{i}\omega)) c^1 <0$.
    Now for the single sinus control~$
    v(t)=\bar v+ c^1 \sin (   \omega  t )$, which is~$T$-periodic with~$T:=\frac{2\pi}{\omega}$, Eq.~\eqref{eq:goe_exp} gives
    $\text{GOE}(v)<0$.
\end{IEEEproof}

\begin{Remark}
Suppose that~$Q$ can be decomposed as~$Q=M^\top M$, with~$M\in\R^{\ell\times n}$. Then the output is~$h(x)=x^\top M^\top Mx =|Mx|_2^2$, and~$G^*(\mathrm{i s})
Q G(\mathrm{i}s)=(MG(\mathrm{i  }s ) )^* M G(\mathrm{i}s)$. 
Therefore, $\text{GOE}(v)$ will be zero for any control in the form~$u(t)=\bar v+c \sin(s t+\eta)$ iff
$MG(\mathrm{i}s)=0$.
%
%
\end{Remark}

\subsubsection{Maximizing the  GOE}
An important application of 
the frequency-based formula for GOE 
in Proposition~\ref{prop:impulse} is that it allows 
to solve optimization problems for~GOE.
To demonstrate this, 
define~$H:\R\to\R^{m\times m}$ by 
\[
H(\omega):=\Re\left( G^*(\mathrm{i}\omega ) Q  
 G(\mathrm{i}\omega) \right) .
 \]
Since~$Q$ is symmetric, so is~$H(\omega)$.
Assume that  
\[
\omega^*:=\arg\max_{\omega>0} \lambda_{\max} (H(\omega)),
\]
exists, and let~$q^*\in\R^m$ be the corresponding eigenvector, normalized such that~$(q^*)^\top q^*=1$.

\begin{proposition}\label{prop:maxi}
    Consider the system~\eqref{eq:intro_miso_system} with  
     a quadratic output function~\eqref{eq:quad},
 and  the 
control~\eqref{eq:multick}. Consider the problem of maximizing~$\text{GOE}(v)$ subject to the  ``energy  constraint''~$\sum_{k\in \mathbb K} (c^k)^\top c^k\leq 1$. 
If~$\lambda_{\max} (H(\omega^*))>0$
then an  optimal solution 
is
\[
v^*(t)=\bar v+q^*\sin(\omega^*t), 
\]
and
\[
\text{GOE}(v^*)=\frac1{2}\lambda_{\max} (H(\omega^*)). 
\]
If~$\lambda_{\max} (H(\omega^*))\leq 0$ then an optimal solution is the constant control~$u(t)=\bar v$.
\end{proposition}

In other words, the optimal strategy is to place all energy into the frequency-direction pair with largest eigenvalue.
\begin{IEEEproof}
    The fact that a single sinus control is sufficient  follows from the additive nature of the formula in~\eqref{eq:goe_exp}. For such a control, the formula becomes
$    
GOE(v) = \frac{1}{2}  
c^\top 
H(\omega) c
$,  and the constraint is~$c^\top c  \leq 1$. Using the fact that~$H(\omega)$ is symmetric implies that the maximum~GOE is attained when~$\omega=\omega^*$ and~$c=q^*$, 
and the maximal value is~$\frac1{2}(q^*)^\top H(\omega^*)q^*= \frac1{2}\lambda_{\max}(H(\omega^*))$. If~$\lambda_{\max}(H(\omega^*))\leq 0$ then  the optimal solution is a constant control. 
\end{IEEEproof}

\section{Applications}
 GOE in the system~\eqref{eq:intro_miso_system} has applications in fields such as electrical engineering, biology, ecology, and process engineering, where an important question is whether temporal fluctuations can improve the average  performance. In this section, we present two   applications that illustrate the theoretical results developed above.

\subsection{Electromagnetic Energy in an RLC Circuit}
      Consider  the series RLC circuit depicted in 
       Figure~\ref{fig:series_rlc} with input voltage~$u(t)$ (in volts).
     Define the state-vector  
\begin{equation}
    x(t):=
    \begin{bmatrix}
        q(t)\\
        i(t)
    \end{bmatrix},
\end{equation}
 where  \(q(t)\) is  the charge on the capacitor, and 
\(i(t)\) is  the current through  the circuit.
Define the   output as the energy stored in the circuit, that is, $y(t)=x^\top (t)Qx(t)$, with~$Q : =
    \begin{bmatrix}
        \frac{1}{2C} & 0\\
        0 & \frac{L}{2}
    \end{bmatrix}
$.
Then the circuit dynamics is described by~\eqref{eq:intro_miso_system}
with~$n=2$, $m=1$, 
\begin{align}\label{eq:ABmatrices_RLC}
    A=
    \begin{bmatrix}
        0 & 1\\
        -\frac{1}{LC} & -\frac{R}{L}
    \end{bmatrix} 
 , 
 \quad
    B=
    \begin{bmatrix}
        0\\
        \frac{1}{L}
    \end{bmatrix},
    \end{align}
    and~$h(z)=z^\top Qz$. Note that~$A$ is Hurwitz. 

Since~$Q$ is positive-definite, $h$ is strictly convex and we conclude that~$\text{GOE}(v)>0$ for any $T$-periodic control~$v$ that generates a nonconstant periodic orbit~$\gamma^v $. In other words, if our goal is to maximize the average energy stored in the circuit then   a nonconstant $T$-periodic voltage input is better than a constant  voltage input with the same
average. Obviously, positive~GOE holds similarly  in many linear electrical networks when the performance measure is the total stored energy.

To derive more explicit results for the~GOE in this circuit,  consider the scalar control
\be\label{eq:sin_control}
v(t)=\bar v +c\sin(\omega t+\phi) ,
\ee
with~$\bar v ,c \in\R$, $\omega>0$, and~$\phi\in[0,2 \pi)$. Note that this is~$T$-periodic with~$T:=2\pi/\omega$.
 \begin{figure}[t]
    \centering
    \begin{circuitikz}[american]
        \draw
        (0,0) to[V, invert, l_={$u(t)$}] (0,3)
              to[R, l={$R$}] (3,3)
              to[L, l={$L$}, i>^={$i(t)$}] (6,3)
              to[C, l_={$C$}] (6,0)
              -- (0,0);

        \node at (6.45,2.15) {$+$};
        \node at (6.45,0.85) {$-$};
        \node at (7.35,1.50) {$v_C(t)$};
    \end{circuitikz}
    \caption{A series RLC circuit driven by a voltage source \(u(t)\).}
    \label{fig:series_rlc}
\end{figure}
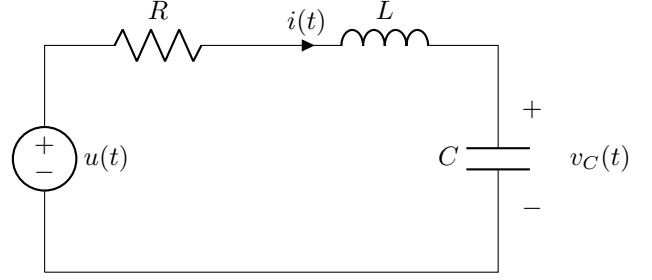

   For the RLC circuit,  the transfer function in~\eqref{eq:def_Gs} is  
$G(s)=\frac{C}{ CL s^2 +C R s +1 }\begin{bmatrix}
    1\\ s
\end{bmatrix}$, so
\[
G^*(\mathrm{i}\omega) Q G (\mathrm{i}\omega) = 
z(\omega) (\frac{1}{2C}+\omega^2\frac{L}{2}),
\]
with~$z(\omega):=\frac{C^2}{(1-CL \omega^2)^2+(CR\omega)^2}$. Using
Proposition~\ref{prop:impulse} gives
\begin{align}\label{eq:expli}
GOE(v)& =    {c^2 z(\omega)}   
   (\frac{1}{4C}+ \frac{L\omega^2}{4})\nonumber
 \\
 &= \frac{c^2}{4} 
 \frac{L+\frac1{C\omega^2} }{R^2+(L\omega-\frac1{C \omega})^2 }.
\end{align}
This implies the following. First,~$\text{GOE}(v)$ is positive, and increases quadratically with the amplitude~$c$ of the sinusoidal signal. Second,
$\text{GOE}(v)$ goes to zero as~$\omega\to\infty$. This makes sense,  
as the circuit filters out high frequencies. Third, 
the explicit expression~\eqref{eq:expli}
can also be used to find the frequency~$\omega^*$ that maximizes~$\text{GOE}(v)$. A calculation shows that~$\omega^*<\omega_0$, where~$\omega_0:=\frac1{\sqrt{L C}}$ is the resonance frequency.
   
As a numerical example, consider 
the circuit   with parameter
values
\[
R=1  \mathrm{\Omega},  \,C=2 F, \,  L= 3 H,
\]
and initial condition~$x(0)=0$.
The resonance frequency is~$\omega_0=1/\sqrt{6}$. 
Figure~\ref{fig:circ_output}
depicts the output~$y(t)=x^\top (t)Qx(t)$ as a function of~$t$ for two controls:
$u(t)=v(t)$ in~\eqref{eq:sin_control}
with~$\bar v=1$, $c=1$, $\omega=\omega_0$, and~$\phi=0$ and the control~$u(t)=\bar v$. 
  It may be seen that for the constant   [periodic] control,~$h( x(t )) $ converges to the constant~$h(e^{\bar v})=1$ [periodic orbit~$\gamma^v$]  and that the time average of~$h(\gamma^v(t))$ is larger than~$1$.
  Figure~\ref{fig:qt_vs_it}
  depicts~$x_1(t)$ vs~$x_2(t)$. It may be seen that~$x(t)$ converges to a periodic solution. The equilibrium~$e^{\bar v}$ is marked by x, and it may be seen that this is the time average  of the periodic orbit. 
\begin{figure}
    \centering
    \includegraphics[scale=0.5]{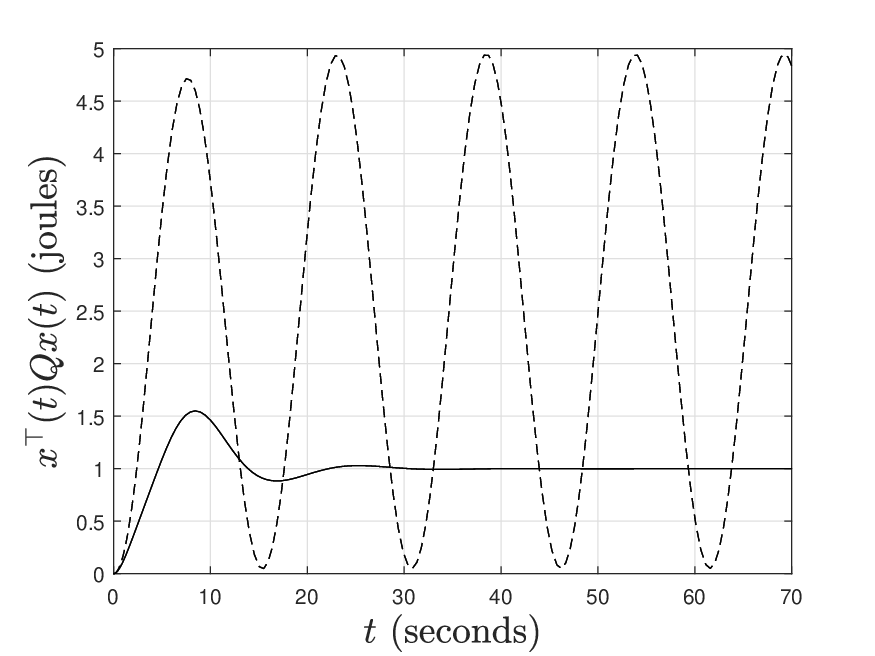}
    \caption{  Energy~$x^\top(t)Qx(t)$  in the RLC circuit as a function of time~$t$  for the control~$u(t)=1$ (solid line) and for the control~$u(t)=1+\sin(\omega_0t)$ [dashed line]. }
    \label{fig:circ_output}
\end{figure}

\begin{figure}
    \centering
    \includegraphics[scale=0.5]{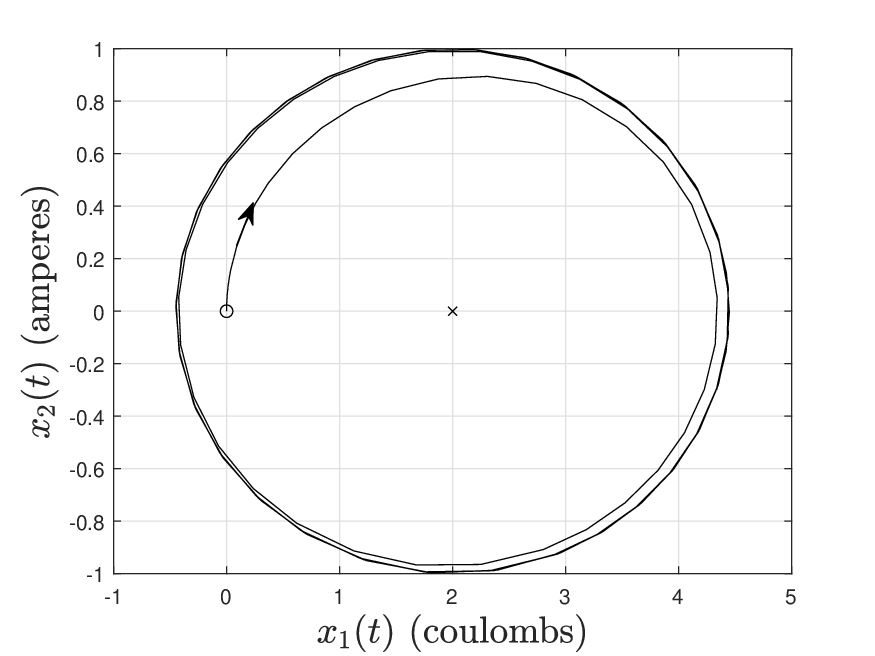}
    \caption{  State variables~$x_1(t)$ vs  $x_2(t)$ in the RLC circuit for the control~$u(t)=1+\sin(\omega_0t)$ and initial condition~$x(0)=0$.
    The value~$e^{\bar v}=\begin{bmatrix}
        2&0
    \end{bmatrix}^\top$ is marked by~x.}
    \label{fig:qt_vs_it}
\end{figure}

\subsection{Drug Response in Pharmacodynamics}
Consider the 
model
\begin{align}\label{eq:pharm_dyn}
    \dot x_1&=u-k x_1,\nonumber\\
    \dot x_2&=kx_1-kx_2 ,\nonumber\\
    &\vdots\\
    \dot x_n&=k x_{n-1}-k x_n,\nonumber
\end{align}
with~$k>0$,~$u(t)\geq 0 $ for all~$t\geq0$,  and~$x(0)\in\R^n_{\geq 0}$. This   transit compartmental model~(TCM)   is ubiquitous in pharmacodynamics (see~\cite{Koch2013} and the references therein).  The control input corresponds to drug administration, whereas the state dynamics model the transport of the drug between physiological compartments, such as the gastrointestinal tract and the bloodstream following oral dosing.
 
We assume that  the drug therapeutic effect takes place at the last compartment, that is, compartment~$n$.  A standard  model   relating drug concentration to its effect~$e$  is the  $\text{E}_{\max}$ model~\cite{Felmlee2012,emax2010}:
\be\label{eq:emax}
e(x)=\frac{\text{E}_{\max} x_n }{\text{EC}_{50} +x_n }. 
\ee
Here, 
$\text{E}_{\max}>0$ is the maximal effect, and~$\text{EC}_{50}>0$ is the concentration needed to produce the half-maximal effect. Eq.~\eqref{eq:emax} describes a saturation effect: the amount of binding possibilities to the drug at the receptor saturates, and more drug concentration will lead to a diminishing effect.

Note that~\eqref{eq:pharm_dyn} is in the form~\eqref{eq:intro_miso_system} with~$m=1$,
$B=\begin{bmatrix}
    1&0&\dots&0
\end{bmatrix}^\top$, and the matrix~$A$ Hurwitz and Metzler, so the system is a positive control system~\cite{farina2000,monotone_control_systems}, and~$\R^n_{\geq0}$ is a forward invariant set for the dynamics.
We  associate with the linear system the output map~$h(x)=e(x)$.
A natural question is whether periodic drug administration has a better therapeutic  effect, on average. Since~$h(x)$ is a   concave function of~$x$ on~$\R^n_{\geq 0}$, our theoretical  results   imply that~$GOE(u)\leq 0 $  for any nonnegative
control~$u$, that is, in this model periodic drug administration   cannot outperform the therapeutic effect relative to constant administration with the same average dose. 

   As a numerical example, consider the case~$n=3$, $k=2$, $\text{E}_{\max}=\text{EC}_{50}=1$, and
   the control~$v(t)=\bar v+c\sin(\omega t)$, with~$\bar v=1$, $c=1/2$. For the constant control~$u(t)=\bar v$ the state converges to~$e^{\bar v}=-A^{-1}B\bar v=\begin{bmatrix}
       1/2&1/2&1/2
   \end{bmatrix} ^\top$ and the output converges to~$h(e^{\bar v})=1/3$. For the control~$u(t)=v(t)$, which is~$T$-periodic, with~$T=2\pi/\omega$, the state converges to a $T$-periodic orbit~$\gamma^u$.
   Figure~\ref{fig:pharma} depicts~$h(e^{\bar v})$ [dashed line] and~$\frac{1}{T}\int_0^T h(\gamma^u(t))\, dt$ [dotted line]
   as a function of~$\omega\in[1,15]$. It may be seen that the output corresponding to the constant control is always strictly larger than the average output along the  periodic orbit.

\begin{figure}
    \centering
    \includegraphics[scale=0.5]{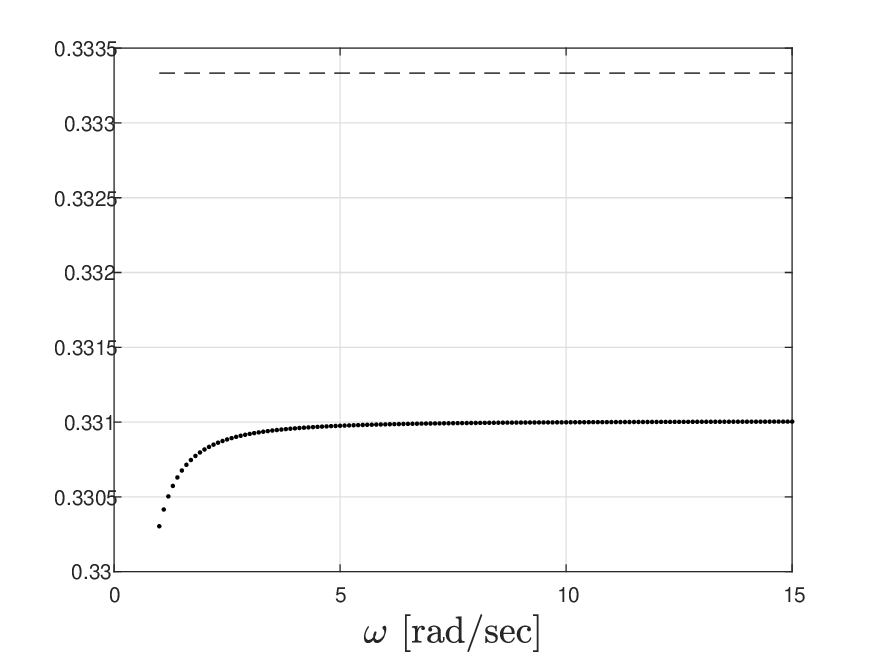}
    \caption{ Output~$h(e^{\bar v})$    [dashed line] vs the averaged output~$\frac{1}{T}\int_0^T h(\gamma^u(t))\, dt$ [dotted line] as a function of~$\omega\in[1,15]$. }
    \label{fig:pharma}
\end{figure}

  \section{Discussion}
A positive GOE  implies that periodic forcing not only synchronizes the system to the excitation, but also improves the average performance relative to a constant input with the same mean value. Understanding when this phenomenon occurs is important in numerous applications including traffic control, biological regulation, and pharmacological interventions.

For linear systems with  a linear output, periodic forcing cannot generate a non-zero  GOE because averaging over a period commutes with both the dynamics and the output map.
We analyzed~GOE in continuous-time
stable linear control systems with  a nonlinear output map, that is, Wiener systems.  Our results show that introducing a nonlinear output   changes the picture. Although the state dynamics remain linear, the average output along the entrained periodic trajectory need not coincide with the output evaluated at the average state. This mismatch is the basic mechanism that generates a nonzero~GOE.

Our results show that  the sign and magnitude 
of~GOE are determined entirely by the geometry of the output function evaluated along the difference between the periodic orbit and the equilibrium.  In particular,  a convex [concave] output map  guarantees  a nonnegative [nonpositive]~GOE for every periodic input.    

The representation of GOE as an average Bregman divergence provides an additional geometric interpretation. From this perspective, GOE measures the average discrepancy between the nonlinear output evaluated along the entrained periodic orbit and its first-order approximation around the equilibrium generated by the averaged input. In special cases, such as the negative entropy output, this representation yields an average Kullback–Leibler divergence, suggesting possible connections between~GOE, information-theoretic quantities, and nonequilibrium statistical mechanics.

For quadratic outputs, the analysis yields an explicit frequency-domain characterization of GOE. This characterization reveals that individual harmonics contribute independently to GOE and provides necessary and sufficient conditions guaranteeing its sign. It also allows to
identify the frequency and direction that maximizes~GOE under an energy constraint. These results connect the notion of~GOE to classical frequency-domain analysis and suggest practical methods for designing efficient periodic excitation protocols.
 
There are many interesting  directions for future research. One is the analysis of interconnected  Wiener systems, where interactions between subsystems may create additional mechanisms affecting~GOE. Another is the study of stochastic forcing and the relationship between~GOE and fluctuations induced by noise. It would also be interesting to extend the analysis to weakly nonlinear systems of the form
\[
\dot x = Ax + Bu + \varepsilon f(x,u),
\]
where now nonlinearities appear both in the dynamics and in the output. More generally, obtaining tractable conditions for analyzing~GOE in nonlinear control systems remains a challenging and important  open problem.

\end{document}